\documentclass[12pt]{amsart}

\usepackage{amssymb}
\usepackage{amscd}

\title{Automorphism groups of smooth plane curves with many Galois points}
\author{Satoru Fukasawa}

\address{Department of Mathematical Sciences,  
Faculty of Science, Yamagata University, 
Kojirakawa-machi 1-4-12, Yamagata 990-8560, Japan.}
\email{s.fukasawa@sci.kj.yamagata-u.ac.jp} 
\newtheorem{theorem}{Theorem}
\newtheorem{proposition}{Proposition}

\newtheorem{lemma}{Lemma}

\theoremstyle{definition} 
\newtheorem{remark}{Remark}

\begin{document}
\begin{abstract} 
We settle the automorphism groups of curves appearing in a classification list of smooth plane curves with at least two Galois points. 
One of them is an ordinary curve whose automorphism group exceeds the Hurwitz bound. 
\end{abstract}
\maketitle
\section{Introduction}  
Let the base field $K$ be an algebraically closed field of characteristic $p=2$ and let $q=2^e \ge 4$. 
We consider smooth plane curves given by
\begin{equation} \label{d-Galois}
Z\prod_{\alpha \in \mathbb F_{q}}(X+\alpha Y+\alpha^2 Z)+\lambda Y^{q+1}=0, \tag{*}
\end{equation}  
and 
\begin{equation} \label{3-Galois}
(X^2+XZ)^2+(X^2+XZ)(Y^2+YZ)+(Y^2+YZ)^2+\lambda Z^4=0 , \tag{**}
\end{equation} 
where $\lambda \in K \setminus \{0, 1\}$. 
These curves appear in the classification list of smooth plane curves with at least two Galois points (\cite[Theorem 3]{fukasawa2}, see \cite{miura-yoshihara, yoshihara} for definition of Galois point). 
The automorphism groups of other curves (Fermat, Klein quartic and the curve $x^3+y^4+1=0$) in the list were studied by many authors (see, for example, \cite{hkt, hurt, ks, ritzenthaler}). 
In this paper, we settle the automorphism groups of these curves, as follows. 

\begin{theorem} \label{inner}
Let $C$ be the plane curve given by $(\ref{d-Galois})$ of degree $q+1$ and genus $g_C=q(q-1)/2$. 
Then, ${\rm Aut}(C) \cong {\rm PGL}(2, \Bbb F_q)$. 
In particular, $|{\rm Aut}(C)|=q^3-q$ and $> 84(g_C-1)$ if $q \ge 64$.   
\end{theorem}

\begin{theorem} \label{outer}
Let $C$ be the plane curve given by $(\ref{3-Galois})$ of degree four. 
Then, ${\rm Aut}(C)$ is isomorphic to the symmetric group $S_4$ of degree four. 
In particular, $|{\rm Aut}(C)|=24$. 
\end{theorem}

It is well known that the order of the automorphism group of any curve with genus $g_C>1$ is bounded by $84(g_C-1)$ in characteristic zero, by Hurwitz. 
Our curve given by $(\ref{d-Galois})$ is an ordinary curve whose automorphism group exceeds the Hurwitz bound (see Remark \ref{ordinary}). 
This is different from examples of Subrao \cite{subrao} and Nakajima \cite{nakajima} by the genera. 

Our theorems are proved by considering the Galois groups at Galois points. 
Therefore, our study is related to the results of Kanazawa, Takahashi and Yoshihara \cite{kty}, Miura and Ohbuchi \cite{mo}. 

\section{Proof of Theorem \ref{inner}} 
According to \cite[Appendix A, 17 and 18]{acgh} or \cite{chang}, any automorphism of smooth plane curves of degree at least four is the restriction of a linear transformation. 
Therefore, we have an injection $$ {\rm Aut}(C) \hookrightarrow {\rm PGL}(3, K). $$ 
Let $L_Y$ be the line given by $Y=0$, and let $P_1=(1:0:0)$ and $P_2=(0:0:1)$. 
A point $P \in \Bbb P^2$ is said to be Galois, if the field extension induced by the projection $\pi_P$ from $P$ is Galois. 
If $P$ is a Galois point, then we denote by $G_P$ the Galois group. 
For $\gamma \in {\rm Aut}(C)$, we denote the set $\{Q \in \Bbb P^2 \ | \ \gamma(Q)=Q\}$ by $L_{\gamma}$.  
We have the following properties for curves with $(\ref{d-Galois})$ (see also \cite{fukasawa2}). 

\begin{proposition} \label{fundamental} 
Let $C$ be the plane curve given by $(\ref{d-Galois})$.  
Then, we have the following. 
\begin{itemize}
\item[(a)] $C \cap L_Y=L_Y(\Bbb F_q)$, where $L_Y(\Bbb F_q)$ is the set of $\Bbb F_q$-rational points of $L_Y$. 
We denote by $L_Y(\Bbb F_q)=\{P_1, \ldots, P_{q+1}\}$. 
\item[(b)] The set of Galois points on $C$ coincides with $L_Y(\Bbb F_q)$. 
\item[(c)] For the projection $\pi_{P_1}$ from $P_1$, the ramification index at $P_1$ is $q$ and there are exactly $(q-1)$ lines $\ell$ such that the ramification index at each point of $C \cap \ell$ is equal to two. 
Furthermore, $\sigma(P_1)=P_1$ for any $\sigma \in G_{P_1}$. 
\item[(d)] If $i, j, k$ are different, then there exists $\sigma \in G_{P_i}$ such that $\sigma(P_j)=P_k$. 
\end{itemize}
\end{proposition}

\begin{proof}
Since the set $C \cap L_Y$ is given by $Y=Z\prod_{\alpha \in \Bbb F_q}(X+\alpha^2Z)=0$, we have (a). 
See \cite[Section 3]{fukasawa1}, \cite[Section 4]{fukasawa2} for (b). 
An automorphism $\sigma \in G_{P_1}$ is given by $(x,y) \mapsto (x+\alpha y+\alpha^2, y)$ for some $\alpha \in \Bbb F_q$ (see \cite[Section 4]{fukasawa2}). 
Then, the set $L_{\sigma}$ coincides with the line defined by $\alpha Y+\alpha^2 Z=0$. 
It follows from \cite[III.8.2]{stichtenoth} that we have (c). 
Since $G_{P_i}$ acts on $C \cap \ell \setminus \{P_i\}$ transitively if $\ell$ is a line passing through $P_i$ by a natural property of Galois extension (\cite[III.7.1]{stichtenoth}), we have (d).  
\end{proof}

We determine ${\rm Aut}(C)$. 

\begin{lemma} \label{injective} 
The restriction map $\gamma \mapsto \gamma|_{L_Y}$ gives an injection 
$$ r:{\rm Aut}(C) \hookrightarrow {\rm PGL}(L_Y(\Bbb F_q)) \cong {\rm PGL}(2, \Bbb F_q). $$
\end{lemma}

\begin{proof}
Let $\gamma \in {\rm Aut}(C)$. 
Since the set of Galois points is invariant under the linear transformation, $\gamma (L_Y(\Bbb F_q))=L_Y(\Bbb F_q)$, by Proposition \ref{fundamental}(a)(b). 
Therefore, $r$ is well-defined. 

Assume that $\gamma|_{L_Y}$ is identity. 
Then, $\gamma(T_{P_i}C)=T_{\gamma(P_i)}C=T_{P_i}C$ and the point given by $T_{P_1}C \cap T_{P_i}C$ is fixed by $\gamma$ for any $i$.  
If $P_i=(\beta:0:1) \in L_Y(\Bbb F_q)$, then $T_{P_i}C$ is given by $X+\sqrt{\beta} Y+\beta Z=0$. 
Since $\gamma|_{T_{P_1}C}$ is an automorphism of $T_{P_1}C \cong \Bbb P^1$ and there are $q$ $(\ge 4)$ points fixed by $\gamma$, $\gamma|_{T_{P_1}C}$ is identity. 
Since $\gamma|_{L_Y}=1$ and $\gamma|_{T_{P_1}C}=1$, $\gamma$ is identity on $\Bbb P^2$. 
\end{proof}

\begin{lemma} Let $H(C):=\{\gamma \in {\rm Aut}(C)|\gamma(P_1)=P_1, \gamma(P_2)=P_2\}$ and let $H_0:=\{\tau \in {\rm PGL}(L_Y(\Bbb F_q))|\tau(P_1)=P_1, \tau(P_2)=P_2\}$. 
Then, $r(H(C))=H_0$. 
In particular, $H_0 \subset r({\rm Aut}(C))$. 
\end{lemma} 

\begin{proof}
We have $r(H(C)) \subset H_0$. 
According to \cite[Lemma 4 and Page 100]{fukasawa2}, $H(C)$ is a cyclic group of order $q-1$. 
We can also prove that $H_0$ is a cyclic group of order at most $q-1$ (see, for example, \cite[Lemma 2(2)]{fukasawa2}). 
Therefore, we have $r(H(C))=H_0$. 
\end{proof}

\begin{lemma} \label{surjective}
The restriction map $r$ is surjective. 
\end{lemma}

\begin{proof}
Let $\tau \in {\rm PGL}(L_Y(\Bbb F_q))$ and let $\tau(P_1)=P_i$, $\tau(P_2)=P_j$. 
We take $k \ne 1, i$. 
By Proposition \ref{fundamental}(d), there exists $\gamma_1 \in r(G_{P_k})$ such that $\gamma_1\tau(P_1)=P_1$. 
Further, by Proposition \ref{fundamental}(c)(d), there exists $\gamma_2 \in r(G_{P_1})$ such that $\gamma_2\gamma_1\tau(P_1)=P_1$ and $\gamma_2\gamma_1\tau(P_2)=P_2$. 
Then, $\gamma_2\gamma_1\tau \in H_0$. 
By Lemma above, $\gamma_2\gamma_1\tau \in r({\rm Aut}(C))$. 
This implies $\tau \in r({\rm Aut}(C))$. 
\end{proof}

We have ${\rm Aut(C)} \cong {\rm PGL}(2, \Bbb F_q)$ by Lemmas \ref{injective} and \ref{surjective}.

\begin{remark} \label{ordinary}
According to Deuring-$\breve{\mbox{S}}$afarevi$\breve{\mbox{c}}$ formula (\cite{subrao}), the $p$-rank $\gamma_C$ of the curve $C$ is computed by ramification indices for the Galois covering $\pi_{P_1}$.  
Using Proposition \ref{fundamental}(c), we have 
$$ \frac{\gamma_C-1}q=(-1)+\left(1-\frac{1}{q}\right)+(q-1)\left(1-\frac{1}{2}\right).$$
This implies $\gamma_C=q(q-1)/2=g_C$, i.e. $C$ is ordinary. 
\end{remark}

\begin{remark} We also have the following for ${\rm Aut}(C)$. 
\begin{itemize}
\item[(a)] $|{\rm Aut}(C)|=g_C\times(3+\sqrt{8g_C+1})$. 
\item[(b)] ${\rm Aut}(C)=\langle G_{P_1}, \ldots, G_{P_{q+1}}\rangle=\langle G_{P_1}, G_{P_2} \rangle$. 
\end{itemize}  
\end{remark}

\section{Proof of Theorem \ref{outer}} 
Similarly to the previous section, we have an injection $$ {\rm Aut}(C) \hookrightarrow {\rm PGL}(3, K). $$ 
Let $L_Z$ be the line given by $Z=0$, and let $P_1=(1:0:0)$, $P_2=(1:1:0)$ and $P_3=(0:1:0)$. 
If $P$ is a Galois point, then we denote by $G_P$ the Galois group. 
For $\gamma \in {\rm Aut}(C)$, we denote the set $\{Q \in \Bbb P^2 \ | \ \gamma(Q)=Q\}$ by $L_{\gamma}$.  
We have the following properties for curves with $(\ref{3-Galois})$ (see \cite[Sections 3 and 4]{fukasawa3}). 

\begin{proposition} \label{fundamental2} 
Let $C$ be the plane curve given by $(\ref{3-Galois})$.  
Then, we have the following. 
\begin{itemize} 
\item[(a)] The set of Galois points in $\Bbb P^2 \setminus C$ coincides with $L_Z(\Bbb F_2)=\{P_1, P_2, P_3\}$. 
\item[(b)] For each $i$, there exists a unique $\sigma_i \in G_{P_i} \setminus \{1\}$ such that $L_{\sigma_i}=L_Z$.    
\item[(c)] There exist exactly two lines $\ell$ such that $\ell \ni P_1$, $\ell \ne L_Z$ and $\ell$ is the tangent line at two points in $C \cap \ell$. 
Conversely, if $\ell$ is such a line, then there exists $\tau \in G_{P_1} \setminus \langle \sigma_1 \rangle$ such that $L_{\tau}=\ell$. 
\item[(d)] There exist exactly four points $Q_1, Q_2, Q_3, Q_4 \in \Bbb P^2 \setminus L_{Z}$ such that the line $\overline{Q_iQ_j}$ which passes through $Q_i, Q_j$ is a tangent line of $C$ for each $i, j$ with $i \ne j$ and $\overline{Q_iQ_j} \ni P_k$ for some $k$. 
Such points are $(0:0:1)$, $(1:0:1)$, $(0:1:1)$ and $(1:1:1)$. 
\end{itemize}
\end{proposition}

\begin{proof} 
For (a)(d), see \cite[Section 4]{fukasawa3}. 
For the sake of readers, we explain (b)(c) for $i=1$.   
Let $\sigma, \tau$ be linear transformations given by 
$$\sigma(X:Y:Z)=(X+Z:Y:Z), \ \tau(X:Y:Z)=(X+Y:Y:Z). $$ 
Then, $G_{P_1}=\{1, \sigma, \tau, \sigma\tau\}$. 
Since $\sigma|_{L_Z}=1$ and $\tau|_{L_Z} \ne 1$, we have (b). 
Note that the line $L_{\tau}$ is given by $Y=0$ and the line $L_{\sigma\tau}$ is given by $Y+Z=0$. 
Referring \cite[III. 8.2]{stichtenoth}, we have (c). 
\end{proof}

First we prove the following. 

\begin{lemma}
Let $X=\{Q_1, Q_2, Q_3, Q_4\}$ and let $S(X)$ be the group of all permutations on $X$. 
Then, there exists an injection ${\rm Aut}(C) \hookrightarrow S(X) \cong S_4$. 
\end{lemma}

\begin{proof}
By Proposition \ref{fundamental2}(d), we have a well-defined homomorphism ${\rm Aut}(C) \rightarrow S(X)$ by $\gamma \mapsto \gamma|_X$. 
If $\gamma \in {\rm Aut}(C)$ fixes $Q_1, Q_2, Q_3, Q_4$, then $\gamma$ fixes $P_1, P_2, P_3$ also. 
Note that $X \cup \{P_1, P_2, P_3\}=\Bbb P^2(\Bbb F_2)$. 
Then, $\gamma$ is identity on the projective plane. 
\end{proof}

We prove that $|{\rm Aut}(C)| \ge 24$. 
Let $H:=\langle \sigma_1, \sigma_2 \rangle$.

\begin{lemma}
The restriction map 
$$r : {\rm Aut}(C) \rightarrow {\rm PGL}(L_Z(\Bbb F_2)) \cong S_3; \ \gamma \mapsto \gamma|_{L_Z} $$ 
is surjective and its kernel coincides with $H$. 
In particular, $|{\rm Aut}(C)| \ge 24$. 
\end{lemma}

\begin{proof}
Let $\gamma \in {\rm Aut}(C)$. 
Since the set of Galois points is invariant under the linear transformation, $\gamma(\{P_1, P_2, P_3\})=\{P_1, P_2, P_3\}$, by Proposition \ref{fundamental2}(a). 
Therefore, $r$ is well-defined. 

We consider the kernel. 
Assume that $\gamma|_{L_{Z}}$ is identity. 
Let $\sigma_i \in G_{P_i}$ be an automorphism as in Proposition \ref{fundamental2}(b) for $i=1,2$ and let $\tau, \eta \in G_{P_1} \setminus \langle \sigma_1 \rangle$ with $\tau \ne \eta$. 
Then, $\gamma(L_{\tau})=L_{\tau}$ or $L_{\eta}$ by Proposition \ref{fundamental2}(c). 
Therefore, $\sigma_2^k\gamma(L_{\tau})=L_{\tau}$, where $k=0$ or $1$. 
Since $\sigma_1$ acts on $C \cap L_{\tau}$, $\sigma_1^l\sigma_2^k\gamma$ is identity on $L_{\tau}$ and $L_Z$, where $l=0$ or $1$. 
This implies that $\sigma_1^l\sigma_2^k\gamma$ is identity on $\Bbb P^2$. 
We have $\gamma \in H$.
 
We prove that $r$ is surjective. 
We have an injection ${\rm Aut}(C)/H \hookrightarrow S_3$. 
Let $\tau_i \in G_{P_i} \setminus \langle \sigma_i \rangle$ for each $i$. 
Since $\tau_1\tau_2(P_1)=P_2$, $\tau_1\tau_2(P_2)=P_3$ and $\tau_1\tau_2(P_3)=P_1$, the order of $\tau_1\tau_2H \in {\rm Aut}(C)/H$ is three. 
Since the group ${\rm Aut}(C)/H$ has elements of order two and three, we have ${\rm Aut}(C)/H=S_3$. 
\end{proof}

We have the conclusion, by these two lemmas. 

\begin{remark}
We also have ${\rm Aut}(C)=\langle G_{P_1}, G_{P_2}, G_{P_3}\rangle=\langle G_{P_1}, G_{P_2} \rangle$. 
\end{remark}

\
\begin{center} {\bf Acknowledgements} \end{center} 
The author was partially supported by JSPS KAKENHI Grant Number 25800002.


\begin{thebibliography}{100}
\bibitem{acgh} E. Arbarello, M. Cornalba, P. A. Griffiths and J. Harris, {\it Geometry of algebraic curves}, Vol. I. Grundlehren der Mathematischen Wissenschaften {\bf 267}, Springer-Verlag, New York, 1985. 
\bibitem{chang} H. C. Chang, On plane algebraic curves, Chinese J. Math. {\bf 6} (1978), 185--189. 
\bibitem{fukasawa1} S. Fukasawa, On the number of Galois points for a plane curve in positive characteristic, III, Geom. Dedicata {\bf 146} (2010), 9--20. 
\bibitem{fukasawa2} S. Fukasawa, Complete determination of the number of Galois points for a smooth plane curve, Rend. Sem. Mat. Univ. Padova {\bf 129} (2013), 93--113.  
\bibitem{fukasawa3} S. Fukasawa, Galois points for a plane curve in characteristic two, J. Pure Appl. Algebra {\bf 218} (2014), 343--353.  
\bibitem{hkt} J. W. P. Hirschfeld, G. Korchm\'{a}ros and F. Torres, {\it Algebraic curves over a finite field}, Princeton Ser. Appl. Math., Princeton Univ. Press, Princeton, 2008.
\bibitem{hurt} N. Hurt, {\it Many rational points}, Kluwer Academic Publishers, Dordrecht, 2003. 
\bibitem{kty} M. Kanazawa, T. Takahashi and H. Yoshihara, The group generated by automorphisms belonging to Galois points of the quartic surface, Nihonkai Math. J. {\bf 12} (2001), 89--99. 
\bibitem{ks} M. J. Klassen and E. F. Schaefer, Arithmetic and geometry of the curve $y^3+1=x^4$, Acta Arith. {\bf 74} (1996), 241--257. 
\bibitem{mo} K. Miura and A. Ohbuchi, Automorphism group of plane curve computed by Galois points, Beitr. Algebra Geom., to appear. 
\bibitem{miura-yoshihara} K. Miura and H. Yoshihara, Field theory for function fields of plane quartic curves, J. Algebra {\bf 226} (2000), 283--294. 
\bibitem{nakajima} S. Nakajima, $p$-ranks and automorphism groups of algebraic curves, Trans. Amer. Math. Soc. {\bf 303} (1987), 595--607. 
\bibitem{ritzenthaler} C. Ritzenthaler, Automorphism group of $C: y^3+x^4+1=0$ in characteristic $p$, JP J. Algebra Number Theory Appl. {\bf 4} (2004), 621--623. 
\bibitem{stichtenoth} H. Stichtenoth, {\it Algebraic function fields and codes}, Universitext, Springer-Verlag, Berlin (1993). 
\bibitem{subrao} D. Subrao, The $p$-rank of Artin-Schreier curves, Manuscripta Math. {\bf 16} (1975), 169--193. 
\bibitem{yoshihara} H. Yoshihara, Function field theory of plane curves by dual curves, J. Algebra {\bf 239} (2001), 340--355. 
\end{thebibliography}
\end{document}